\documentclass[a4paper,11pt]{article}
\usepackage{graphics,color}
\usepackage{amsmath}
\usepackage{amssymb}
\usepackage{mathrsfs}
\usepackage{cite}
\usepackage{verbatim}
\usepackage{float}
\usepackage{graphicx}
\usepackage{amsthm}
\usepackage{textcomp}
\usepackage{subfig}
\usepackage{esint}
\usepackage{enumerate}
\usepackage{hyperref}

\numberwithin{equation}{section}
\mathchardef\emptyset="001F

\newtheorem{Theorem}{Theorem}[section]
\newtheorem{Definition}[Theorem]{Definition}
\newtheorem{Proposition}[Theorem]{Proposition}
\newtheorem{Corollary}[Theorem]{Corollary}
\newtheorem{Lemma}[Theorem]{Lemma}

\newcommand{\nada}[1]{}

\newcommand{\mres}{\mathbin{\vrule height 1.6ex depth 0pt width
0.13ex\vrule height 0.13ex depth 0pt width 1.3ex}}
\newcommand{\N}{\numberset{N}} 

\newcommand{\numberset}{\mathbb}
\newcommand{\Om}{\Omega} 
\newcommand{\R}{\numberset{R}}

\theoremstyle{definition}

\newtheorem{Example}[Theorem]{Example}

 \title{Existence, uniqueness and characterisation of vector-valued absolute minimisers for a second order $L^\infty$-variational problem
}
\author{
Simone Carano,  Nikos Katzourakis, and Roger Moser }
\date{}

\begin{document}
\maketitle

\begin{abstract} 
We study a vectorial $L^\infty$-variational problem of second order, where the supremal functional depends on the vector function $u$ through a linear elliptic operator in divergence form. We prove existence and uniqueness of the minimiser $u_\infty$ under prescribed Dirichlet boundary conditions, together with a characterisation of $u_\infty$ as solution of a specific system of PDEs. Our result can be seen as a twofold extension of the one in \cite{KM}: we generalise it to the vectorial setting and, at the same time, we consider more general elliptic operators in place of the Laplacian.
\end{abstract}
{\bf 2020 MSC:} 49K20; 35A15; 35B38; 35J47; 49J27; 49J45.\\
{\bf Key words and phrases:} Vectorial Calculus of Variations in $L^\infty$; higher order problems; Euler-Lagrange equations; elliptic systems; Gamma Convergence.

%%%%%%%%%%%%%%%%%%%%%%%%%%%%%%%%%%%%%%%%%%%%%%%%%%%%%%%%%%%%%%%%%%%%%%%%%
\section{Introduction}\label{sec:introduction}
%%%%%%%%%%%%%%%%%%%%%%%%%%%%%%%%%%%%%%%%%%%%%%%%%%%%%%%%%%%%%%%%%%%%%%%%

For $n,N\in\N$, let $\Om\subset\R^n$ be an open bounded set and $F:\Om\times\R^N\to\R$ be $\mathscr{L}(\Om)\otimes\mathscr{B}(\R^N)$-measurable, where $\mathscr{L}(\Om)$ is the Lebesgue $\sigma$-algebra of $\Om$ and $\mathscr{B}(\R^N)$ is the Borel $\sigma$-algebra of $\R^N$. Consider the supremal functional
\begin{align}\label{E infty}
E_{\infty}(u):=\|F(\cdot, \mathrm{L}u)\|_{L^\infty(\Om)},
\end{align}
where $u$ belongs to the Fréchet Sobolev space
$$
\mathcal{W}^{2,\infty}(\Om,\R^N):=\bigcap_{1<p<\infty}\Big\{u\in W^{2,p}(\Om,\R^N):\mathrm{L} u\in L^\infty(\Om,\R^N)\Big\},
$$
and $\mathrm L$ is a linear elliptic operator in divergence form, i.e.
$$\mathrm{L}u:=\mathrm{div}(A\mathrm{D}u),$$
for a given symmetric $4$-tensor field $A\in L^{\infty}\left(\Om,(\R^{N\times n})_s^{\otimes2}\right)$. Equivalently, in index notation we can write
$$
(\mathrm{L}u)_i=\mathrm{D}_\alpha(A^{\alpha\beta}_{ij}\mathrm{D}_\beta u^j), \quad\forall i=1,\ldots,N,
$$
where we implicitely sum over the repeated indices $\alpha,\beta\in\{1,\ldots,n\}$, $j\in\{1,\ldots,N\}$. Ellipticity is understood either in the Legendre sense, meaning that there exists $\lambda>0$ such that
$$
AX:X\geq\lambda|X|^2,\quad\forall X\in\R^{N\times n},
$$
or in the Legendre-Hadamard sense, meaning that there exists $\lambda>0$ such that
$$
A(\xi\otimes\eta):(\xi\otimes\eta)\geq\lambda|\xi|^2|\eta|^2,\quad\forall \xi\in\R^N,\quad\forall\eta\in\R^{n}.
$$
These require the function $X\mapsto AX:X$ to be either convex or rank-one convex, respectively.\\
In the above, the symbol \textquotedblleft:" denotes the scalar product between matrices, more explicitely for any $X,Y\in\R^{N\times n}$ we write $AX:Y:=A^{\alpha\beta}_{ij}X^i_\alpha Y^j_\beta$, and $|X|$ stands for the Frobenius norm of $X$. Notice that the Legendre condition is stronger than the Legendre-Hadamard one. Moreover, $A$ is symmetric if $A^{\alpha\beta}_{ij}=A^{\beta\alpha}_{ji}$, or equivalently if
$$
AX:Y=X:AY \quad\forall X,Y\in\R^{N\times n}.
$$
The goal herein is to prove existence and uniqueness of global minimisers of \eqref{E infty} given first order Dirichlet boundary conditions on $\partial \Om$. In addition, we will show that they satisfy a certain system of PDE's. This will generalise the main result in \cite{KM}, achieved by the second and third appearing author, where they considered the scalar case involving the Laplace operator, namely when $N=1$ and $A=I_{n\times n}$. \\
Before entering into the details of our main result, we briefly describe the mathematical context of our problem. We are dealing with a vectorial second order problem in the $L^\infty$-Calculus of Variations, a field initiated by Aronsson in \cite{A}. His pioneering work on the scalar first order case has been well developed by now, and most challenges have been thoroughly analysed and understood (see \cite{Kbook} for a survey reference). More recently, in \cite{K1}, the second author approached the vectorial case, which turned out to be more involved and intriguing than the scalar one. For this reason, a complete theory is still far from reach. The situation is very similar in the higher order case, recently started in \cite{KM,KP2}.  Without any pretension of being exhaustive, we refer also to \cite{K1,K3,K4,KM1} for a glimpse of the literature on vectorial first order problems and to \cite{DK,KP1,KM2,KM3} for higher order ones. Furthermore, some interesting work in this field, relevant to the content of this paper, appears in \cite{AP,BDP,BK,KZ,MWZ,PWZ,PZ,RZ,RZ1}.\\
Now we proceed to specify the setting of our problem.
Let us first state the assumptions on the function $F$. We assume that $F:\Om\times\R^N\to\R$ is a Carathéodory function such that, for almost every $x\in\Om$, $F(x,0)=0$, $F(x,\cdot)$ is of class $C^2$. Moreover, we suppose that there exists $c>0$ such that
 \begin{align}
&{}{\mbox{ for a.e. }x\in\Om,\quad\xi\mapsto F(x,\xi)-\frac{|\xi|^2}{c} \mbox{ is convex on } \R^N},\label{ass0}\\
&F(x,\xi)\geq\frac{|\xi|^2}{c},\quad\mbox{a.e. } x\in\Om,\quad \forall \xi\in\R^N,\label{ass1}\\
&|F_\xi(x,\xi)|\leq c|\xi|,\quad\mbox{a.e. } x\in\Om,\quad \forall \xi\in\R^N,\label{ass2}
\end{align}
 %\begin{align}\label{ass3}&\qquad\qquad F_\xi(x,\xi)\cdot\xi\geq\frac{F(x,\xi)}{c} \quad\forall x\in\Om\quad \forall \xi\in\R^N;\end{align}
 where we denote by $F_\xi$ the gradient of $F$ with respect to the second variable. {}{We note that the convexity of $F(x,\cdot)-\frac{|\cdot|^2}{c}$ is a condition commonly known as \textit{strong convexity} in the literature.} In \cite{KM}, the authors consider a slightly different $F$, however, the particular choice of $F$ is not essential for our result. Indeed, for a function $g:\R^+\to\R^+$ lower semicontinuous and strictly increasing, set $G=g\circ F$ and $E'_\infty(u):=\|G(\cdot,\mathrm{L} u)\|_{L^\infty(\Om)}$. Then, it is not difficult to see that $E_\infty$ and $E'_\infty$ share the same minimisers. So, even if our assumptions on $F$ may seem to be restrictive, our result actually apply to a wider class of functions. This class includes, for instance, functions that are level-convex with sublinear growth, but non-convex and non-smooth at $0$, such as $G(\xi)=|\xi|^{1/2}$, or jumping level-convex functions, like $G(\xi)=\varphi(|\xi|)$ where $\varphi(0)=0$ and $\varphi(t):=\frac{t+n}{2}$ for $n-1<t\leq n$, $n\in\N$.\\
\begin{Example}
 Examples of functions that satisfy our assumptions are given by
{}{
\begin{itemize}
\item $F(x,\xi)=\alpha(x)|\xi|^2+g(x,\xi)$, with $\alpha,\frac{1}{\alpha}\in L^\infty(\Om,\R^+)$ and $g(x,\cdot)$ any nonnegative convex function with at most quadratic growth, such that $g(x,0)=0$;
\item $F(x,\xi)=\alpha(x)\bar F(\xi)$, with $\alpha$ as above and $\bar F$ the square of the cone function associated to $\partial C$, for any uniformly convex set $C\subset\R^N$ of class $C^2$.
\end{itemize}
}
\end{Example}
\noindent
Concerning the symmetric $4$-tensor $A$, we will assume one of the following conditions:
\begin{itemize}
\item[(H1)] $A\in C^{1}(\overline\Om,(\R^{N\times n})_s^{\otimes2})$ and satisfies the Legendre ellipticity condition;
\item[(H2)] $A$ has constant coefficients and satisfies the Legendre-Hadamard ellipticity condition.
\end{itemize}
As it is well known (see e.g. \cite{Dac}), (H2) implies that $X\mapsto AX:X$ is quasi-convex in $\R^{N\times n}$.\\
We consider boundary data $u_0\in\mathcal{W}^{2,\infty}(\Om,\R^N)$ for the variational problem associated to \eqref{E infty}, and we write $$\mathcal{W}_{u_0}^{2,\infty}(\Om,\R^N)=u_0+\mathcal{W}_0^{2,\infty}(\Om,\R^N)$$ for the class of competitors, where $\mathcal{W}_0^{2,\infty}(\Om,\R^N):=\mathcal{W}^{2,\infty}(\Om,\R^N)\cap{W}_0^{2,2}(\Om,\R^N)$.\\
\indent
Finally, we require two more technical assumptions. We need $\mathrm{L}$ to satisfy the following version of the {\it unique continuation property}.
\begin{Definition}{\bf (Unique continuation property)}\label{def ucp}
We say that the operator $\mathrm{L}$ has the unique continuation property if the zero nodal set of non-trivial weak solutions to $\mathrm{L}u=0$ is a null set w.r.t. the Lebesgue measure on $\Om$. 
\end{Definition}
\noindent
This measure theoretical version of unique continuation property has been considered for instance in \cite{FG}.
The reader may be more familiar with related, although different, notions, namely the {\it strong} and {\it weak} unique continuation property, which are often used in the literature. We refer to Section \ref{ucp} for further details and examples of operators L fulfilling the condition in Definition \ref{def ucp}. Moreover, the precise notion of weak solution to elliptic systems in divergence form is recalled in Section \ref{subsec systems}.\\
\indent
In addition, we assume that $u_0$ is $C^2$ on $\partial\Om$ in the Whitney sense, meaning that it posseses uniform second order Taylor expansions (\cite[Sec. 3, p. 64]{W}). In particular, $u_0|{_{\partial\Om}}$ can be extended to a $C^2$-map on $\R^n$, by the Whitney Extension Theorem.

Now we can state our main results.
{}
{\begin{Theorem}[{\bf Existence}]\label{existence thm}
Let $\Om\subset\R^n$ be a bounded open set and let $F:\Om\times\R^N\to\R$ be a Carathéodory function satisfying \eqref{ass1} and \eqref{ass2}, with $F(x,\cdot)$ convex for a.e. $x\in\Om$. Assume that $A\in C^{1}(\overline\Om,(\R^{N\times n})_s^{\otimes2})$ satisfies the Legendre-Hadamard ellipticity condition. Let $E_\infty$ be the functional in \eqref{E infty} and $u_0\in\mathcal{W}^{2,\infty}(\Om,\R^N)$. Then the problem
\begin{align}\label{problem E infty}
e_\infty:=\min_{u\in\mathcal{W}_{u_0}^{2,\infty}(\Om,\R^N)}E_\infty(u)
\end{align}
admits a solution $u_\infty\in\mathcal{W}_{u_0}^{2,\infty}(\Om,\R^N)$.
\end{Theorem}}

\begin{Theorem}[{\bf Uniqueness and characterisation}]\label{main thm}
Let $\Om\subset\R^n$ be a bounded open set with $C^{2}$ boundary and let $F:\Om\times\R^N\to\R$ be a function of class $C^2$ with $F(x,0)=0$ and $F(x,\cdot)$ strongly convex for a.e. $x\in\Om$, satisfying also \eqref{ass1} and \eqref{ass2}. Assume that $A$ is a symmetric $4$-tensor field satisfying either (H1) or (H2). Assume also that the operator $\mathrm{L}$ has the unique continuation property. Let $u_0\in\mathcal{W}^{2,\infty}(\Om,\R^N)$ be such that ${u_0}|_{\partial\Om}\in C^2$ in the sense of Whitney
 and $E_\infty$ be the functional in \eqref{E infty}. Then {}{the minimiser $u_\infty\in\mathcal{W}_{u_0}^{2,\infty}(\Om,\R^N)$ provided in Theorem \ref{existence thm} is the unique solution to \eqref{problem E infty}.} \\
Moreover, a system of PDEs can be derived as a necessary and sufficient condition for the minimality of $u_\infty$. More precisely, if $e_\infty=0$, then $F(\cdot,\mathrm{L}u_\infty)=0$ almost everywhere in $\Om$. If $e_\infty>0$, then there exists a map $f_\infty\in L^1(\Om,\R^N)$ that further belongs to $W^{2,q}_{loc}(\Om,\R^N)$ for all $q<\infty$ and such that
\begin{equation}\label{PDE u infty}
\left\{
\begin{aligned}
&F(x,\mathrm{L} u_\infty(x))\frac{F_\xi(x,\mathrm{L} u_\infty(x))}{|F_\xi(x,\mathrm{L} u_\infty(x))|}=e_\infty\frac{f_\infty(x)}{|f_\infty(x)|},\quad \mbox{a.e. } x\in\Om,\\
&\mathrm{L} f_\infty(x)=0, \quad \mbox{a.e. } x\in\Om.
\end{aligned}
\right.
\end{equation}
In particular, $F(\cdot,\mathrm{L} u_\infty)=e_\infty$ almost everywhere on $\Om$.\\
\end{Theorem}
We point out that the assumptions in Theorem \ref{existence thm} are milder than the ones in Theorem \ref{main thm}.
In the latter, the $C^2$ regularity hypotheses on $\partial\Om$ and $u_0$ are required for the methods to work. Even though we are not able to remove these assumptions at the present stage, we believe that they are of a technical nature.

A common phenomenon in $L^\infty$-variational problems is that global minimisers may not minimise the same functional on subdomains with respect to their own boundary conditions. This non-locality feature naturally suggests to consider the so-called \textit{absolute minimisers}, whose minimality notion incorporates the locality property by definition (see \cite{A}). Explicitly, set $E_\infty(u,U):=\|F(\cdot,\mathrm L u)\|_{L^\infty(U)}$, for $U$ open subset of $\Om$. We say that $u\in\mathcal{W}_{u_0}^{2,\infty}(\Om,\R^N)$ is an absolute minimiser of $E_\infty$ if for every open subset $\Om'\subseteq\Om$ there holds
$$
E_\infty(u_\infty,\Om')\leq E_\infty(u_\infty+\Phi,\Om')\quad\forall\Phi\in\mathcal{W}^{2,\infty}_0(\Om',\R^N).
$$
In our specific problem, however, this notion does not have a central role, since we prove that a global minimiser is unique. Typically, uniqueness properties of the minimiser, together with a PDE characterisation, ensure locality properties. More precisely, the following result is straightforward.
\begin{Corollary}
The global minimiser $u_\infty$ given by Theorem \ref{main thm} is also the unique absolute minimiser of $E_\infty$. In particular, the PDE system \eqref{PDE u infty} characterises absolute minimisers.
\end{Corollary}
\begin{proof}
For the sake of contradiction, assume that there exist an open set $\Om'\subset\Om$ and $\Phi\in\mathcal{W}^{2,\infty}_0(\Om',\R^N)\setminus\{0\}$ such that 
\begin{align}\label{contrad arg}
E_\infty(u_\infty+\Phi,\Om')<E_\infty(u_\infty,\Om').
\end{align}
Therefore, since $u_\infty+\Phi=u_\infty$ in $\Om\setminus\Om'$, we infer that $u_\infty+\Phi$ is a global minimiser of $E_\infty$.
By the uniqueness property of $u_\infty$, we must have $\Phi\equiv0$, contradicting \eqref{contrad arg}. 
\end{proof}

Notice that in the previous proof we could have argued in a similar way by using, instead of the uniqueness, the fact that a global minimiser $u$ satisfies $F(\cdot,\mathrm Lu)=e_\infty$. In fact, the system of PDEs \eqref{PDE u infty} characterises minimisers of \eqref{E infty}. The presence of discontinuous data and unbalanced boundary conditions make the system difficult to solve. However, it gives precise information on the behaviour of the supremal functional along $u_\infty$, forcing it to be constant almost everywhere in $\Om$, namely, the supremum in \eqref{E infty} is attained at almost every point. This is in direct contrast to the scalar first order case \cite{A}.\\
\indent
 We prove existence of minimisers via $L^p$-approximation, which is a very well known technique and is present also in Aronsson's works. Our arguments utilise elliptic estimates derived from the divergence form of the operator $\mathrm{L}$. An interesting direction to explore would be to relax the pure variational nature of $\mathrm{L}$, introducing also lower order terms. The ellipticity plays a fundamental role also in the derivation of the PDE system satisfied by $u_\infty$, in particular to show existence and non-triviality of the map $f_\infty$, which appears with $u_\infty$ as an unknown of the system. Finally, we will see that the uniqueness property of $u_\infty$ will arise from the fact that it satisfies system \eqref{PDE u infty}. It is worth mentioning that the uniqueness argument is significantly different from the one in \cite{KM} and is largely inspired by a technique used in \cite{M}: we study a penalised version of the Dirichlet functional, which ensures the convergence of the approximating $L^p$-minimisers to the preselected $L^\infty$-minimiser.\\
\indent
We stress that our result is new even in the scalar case, since we consider supremands depending on more general elliptic operators than the Laplacian (see also \cite{KM2}), as much as regarding the vectorial case for supremands involving the Laplacian only. As a subcase of the latter, the 1-dimensional setting has a clear physical interpretation (compare \cite{KP2,M,MS}): given initial and final position and velocity, we look for the curve with least maximal magnitude of acceleration. From Theorem \ref{main thm}, we can say that the existing curve is unique and has constant acceleration up to (at most) a single point, where the acceleration  switches sign, that is the only point where the curve is not twice differentiable, in complete analogy with \cite[Theorem 8.1]{KP2}. \\

\section{Preparatory tools}\label{prep}
\subsection{Estimates for elliptic systems}\label{subsec systems}
In this first preliminary section, we collect some useful elliptic estimates in the context of systems of PDEs. For the sake of brevity, we introduce the following notation: for open sets $\Om'\subset\Om\subset\R^n$ and a function $v\in W^{k,p}(\Om,\R^N)$, $k\in\N$, $ 1\leq p\leq\infty,$ we will write
\begin{align*}
&\|v\|_{k,p,\Om'}:=\|v\|_{W^{k,p}(\Om',\R^N)},\\
&\|v\|_{k,p}:=\|v\|_{k,p,\Om};
\end{align*}
if $v\in L^p(\Om,\R^N)$, we write
\begin{align*}
&\|v\|_{p,\Om'}:=\|v\|_{L^p(\Om',\R^N)},\\
&\|v\|_{p}:=\|v\|_{p,\Om}.
\end{align*}
Let $A\in L^\infty_{loc}(\Om,(\R^{N\times n})_s^{\otimes 2})$ and $\Phi\in L^2_{loc}(\Om,\R^{N\times n})$. We recall that $u\in W^{1,2}_{loc}(\Om,\R^N)$ is a weak solution to $\mathrm{L}u=\mathrm{div}(\Phi)$ if
$$
\int_\Om A\mathrm{D} u :\mathrm{D} h=\int_\Om \Phi :\mathrm{D} h\quad\forall h\in C^{\infty}_c(\Om,\R^N).
$$
If $A\in W^{1,\infty}_{loc}(\Om,(\R^{N\times n})_s^{\otimes 2})$ and $\varphi\in L^2_{loc}(\Om,\R^N)$, we say that $u$ is a strong solution to $\mathrm{L}u=\varphi$ if $u\in W^{2,2}_{loc}(\Om,\R^N)$ and $\mathrm{L}u(x)=\varphi(x)$ for a.e. $x\in\Om$.
\begin{Theorem}{\bf (Global $W^{1,p}$-estimates)}\label{Lp est}
Let $\Om$ be an open bounded subset of $\R^n$ with Lipschitz boundary. Let $\Phi\in L^{p}(\Om,\R^{N\times n})$, with $2\leq p<\infty$, and $A\in C^0\left(\overline\Om,(\R^{N\times n})_s^{\otimes2}\right)$ be a symmetric $4$-tensor field satisfying the Legendre-Hadamard ellipticity condition. Assume that $u\in W_0^{1,2}(\Om,\R^N)$ is a weak solution to $\mathrm{L}u=\mathrm{div}(\Phi)$ in $\Om$. Then $u\in W^{1,p}_0(\Om,\R^N)$ and there exists $C>0$ depending only on $p,\Om, \lambda, \|A\|_{\infty}$, such that
$$
\|u\|_{1,p}\leq C\|\Phi\|_p.
$$
\end{Theorem}
\begin{proof}
For the proof we refer to \cite[Theorem 3.29]{ACM}. See also \cite[Section 7.7.1]{GM}.
\end{proof}

\begin{Corollary}{\bf (Calderon-Zygmund estimates)}\label{CZ}
Let $\Om$ be an open bounded subset of $\R^n$ with Lipschitz boundary. Let $\varphi\in L^{p}(\Om,\R^N)$, with $2\leq p<\infty$, and $A\in C^1\left(\overline\Om,(\R^{N\times n})_s^{\otimes2}\right)$ be a symmetric $4$-tensor field satisfying the Legendre-Hadamard ellipticity condition. Assume that $u\in W_0^{2,2}(\Om,\R^N)$ is a strong solution to $\mathrm{L}u=\varphi$ in $\Om$. Then $u\in W_0^{2,p}(\Om,\R^N)$ and there exists $C>0$ depending only on $p,\Om, \lambda, \|A\|_{1,\infty}$, such that
$$
\|u\|_{2,p}\leq C\|\varphi\|_p.
$$
\end{Corollary}
\begin{proof}
{}{
First, we notice that we can write
$\varphi=\mathrm{div}\Phi$, for some $\Phi\in L^{p}(\Om,\R^{N\times n})$ and $\|\Phi\|_{p}\leq C(n,p,\Om)\|\varphi\|_p$. Indeed, consider a cube $[-L,L]^n$ with $L$ so large that $\pi_i(\Om)\subset[-L,L]^n$ for every $i=1,\ldots,n$, where $\pi_i(x):=(x_1,\ldots,x_{i-1},0,x_{i+1},\ldots,x_n)$. Extend $\varphi$ trivially on $[-L,L]^n\setminus\Om$ and notice that the restriction of $\varphi$ on almost every line parallel to the cartesian axes is still $L^p$. Thus, for $i=1,\ldots,n$, $j=1,\ldots,N$, and almost every $x\in[-L,L]^n$ define
$$\Phi_{ij}(x):=\frac{1}{n}\int_{0}^{x_i}\varphi_j(x_1,\ldots,x_{i-1},t,x_{i+1},\ldots,x_n)dt. $$ 
The matrix field $\Phi=(\Phi_{ij})$ satisfies div$\Phi=\varphi$ almost everywhere in $[-L,L]^n$ and 
\begin{align}\label{Phi est}
\|\Phi\|_{L^{p}(\Om,\,\R^{N\times n})}\leq \|\Phi\|_{L^{p}([-L,L]^n, \,\R^{N\times n})}\leq C(n,p,L)\|\varphi\|_{L^{p}([-L,L]^n,\, \R^{N})}= C(n,p,L)\|\varphi\|_{L^{p}(\Om,\, \R^{N})}.
\end{align}
Now, for any $\gamma\in\{1,\ldots,n\}$, let $\mathrm D_\gamma$ be the derivative in the direction $e_\gamma$ and derive the equation $\mathrm Lu=\varphi$, obtaining
\begin{align}\label{eq deriv}
\mathrm{D}_\gamma\varphi=\mathrm{D}_\gamma(\mathrm{L}u)=\mathrm{div}(\mathrm{D}_\gamma A\mathrm{D}u)+\mathrm{div}(A\mathrm{D}(\mathrm{D}_\gamma u))
\end{align}
in the sense of distributions.\\
Let us suppose that $n\geq 3$.
Since $u\in W^{2,2}_0(\Om,\R^N)$, by Sobolev embeddings, we have $\mathrm Du\in L^{2^*}(\Om,\R^{N\times n})$, where $2^*=\frac{2n}{n-2}$. Set $q=\min\{2^*,p\}$ and $\Psi:=\varphi\, e_\gamma-\mathrm D_\gamma A \mathrm Du$, then $\Psi\in L^q(\Om,\R^{N\times n})$. Now, by \eqref{eq deriv}, the function $\mathrm D_\gamma u\in W^{1,2}_0(\Om,\R^{N \times n})$ solves (weakly)
$$
\mathrm L (\mathrm D_\gamma u)=\mathrm{div} \Psi,
$$
so we can apply Theorem \ref{Lp est} with exponent $q$, obtaining
\begin{align}\label{est partial}
\|\mathrm{D}(\mathrm{D}_\gamma u)\|_{q}\leq C\|\Psi\|_q\leq C(\|\mathrm{D}u\|_{q}+\|\varphi\|_p),
\end{align}
where $C=C(n, p,\Om,\lambda,\|A\|_{1,\infty})$.
Moreover, since $\mathrm Lu=\mathrm{div}\Phi$, again by Theorem \ref{Lp est}, we have $$\|\mathrm{D}u\|_q\leq C\|\Phi\|_p\leq C\|\varphi\|_p,$$ where we used also \eqref{Phi est}. Thus, estimate \eqref{est partial} becomes
$$
\|\mathrm{D}(\mathrm{D}_\gamma u)\|_{q}\leq C\|\varphi\|_p.
$$
So we get $\mathrm{D}u\in L^{q^*}(\Om,\R^{N\times n})$ and we can apply again Theorem \ref{Lp est} with $q^*$. A bootstrap argument and the Poincaré inequality give the conclusion.\\
Finally, assume $n\leq 2$. By Sobolev embeddings, we have $\mathrm Du\in L^{q}(\Om,\R^{N\times n})$ for all $q\in[1,\infty)$, in particular for $q=p$. So, we may proceed by applying the same reasoning directly for the exponent $p$, without further iterations.
}
\end{proof}

\begin{Theorem}{\bf (Interior $W^{1,p}$-estimates I)}\label{interior 1,p est I}
Let $\Om,\Phi,$ and $A$ be as in Theorem \ref{Lp est}. Assume that $u\in W^{1,2}(\Om,\R^N)$ is a weak solution to $\mathrm{L}u=\mathrm{div}(\Phi)$ in $\Om$. Then $u\in W^{1,p}_{loc}(\Om,\R^N)$ and for every $\Om'
\Subset\Om$ there exists $C>0$ depending only on $p,\Om,\Om', \lambda, \|A\|_{\infty}$, such that
$$
\|\mathrm{D}u\|_{p,\Om'}\leq C(\|\Phi\|_p+\|\mathrm{D}u\|_2).
$$
\end{Theorem}
\begin{proof}
For a proof we refer to \cite[Theorem 7.2]{GM}.
\end{proof}

\begin{Corollary}{\bf (Interior $W^{2,p}$-estimates I)}\label{interior 2,p est I}
Let $\Om$, $\varphi$, and $A$ be as in Corollary \ref{CZ}. Assume that $u\in W^{2,2}(\Om,\R^N)$ is a strong solution to $\mathrm{L}u=\varphi$ in $\Om$. Then $u\in W^{2,p}_{loc}(\Om,\R^N)$  and for every  $\Om'\Subset\Om$ there exists $C>0$ depending only on $p,\Om, \Om',\lambda, \|A\|_{1,\infty}$, such that
$$
\|u\|_{2,p,\Om'}\leq C(\|\varphi\|_p+\|u\|_{2,2}).
$$
\end{Corollary}
\begin{proof}
For $\gamma=1,\ldots,n$, consider the difference quotient $$\mathrm{D}_\gamma^h u:=\frac{1}{h}(u(\cdot+he_\gamma)-u(\cdot)).$$ 
If $|h|$ is sufficiently small, $\mathrm{D}_\gamma^h u$ is well defined in $\Om'$ and it satisfies the system
$$
\mathrm{D}_\gamma^h\varphi=\mathrm{div}(\mathrm{D}_\gamma^h A\mathrm{D}u)+\mathrm{div}(A\mathrm{D}(\mathrm{D}_\gamma^h u)).
$$
Arguing as in the proof of Corollary \ref{CZ}, by setting $q=$min$\{2^*,p\}$, we can apply Theorem \ref{interior 1,p est I} and obtain
\begin{align*}
\|\mathrm{D}(\mathrm{D}\gamma^h u)\|_{q,\Om'}\leq C(\|\mathrm{D}u\|_{q}+\|\varphi\|_p)\leq C(\|u\|_{2,2}+\|\varphi\|_p).
\end{align*}
A bootstrap argument and \cite[Proposition 4.8]{GM} yields the conclusion.
\end{proof}

It is particularly important for our analysis to deal with distributional solutions to elliptic systems of the form $\mathrm{L}f=\mathrm{div}G$. To define them, it is enough to require that $f\in L^1_{loc}(\Om,\R^N)$, $G\in L^1_{loc}(\Om,\R^{N\times n})$, and $A\in W^{1,\infty}_{loc}(\Om,(\R^{N\times n})_s^{\otimes 2})$. Indeed, we say that $\mathrm{L}f=\mathrm{div}G$ in the sense of distributions on $\Om$ if
$$
\int_\Om f\cdot \mathrm{L}\varphi=-\int_\Om G:\mathrm{D} \varphi\quad\forall \varphi\in C^\infty_c(\Om,\R^N).
$$
Note that the symmetry of $A$ ensures that the definition above is compatible with the weak one.\\
  In the next proposition, we show that with slightly more regular data every distributional solution is actually a weak solution. Our proof is largely inspired to \cite[Lemma 6]{KM2} (see also \cite{BKR} for a more general theory).

\begin{Proposition}{\bf (From distributional to weak solutions)}\label{distrib to weak}
Let $\Om$ be an open bounded subset of $\R^n$ with Lipschitz boundary. Assume either that $A\in C^1\left(\Om,(\R^{N\times n})_s^{\otimes2}\right)$ satisfies the Legendre condition or that $A$ has constant coefficients and satisfies the Legendre-Hadamard condition. Suppose that $f\in L^1(\Om,\R^N)$ and $G\in L^2(\Om,\R^{N\times n})$ are such that $\mathrm{L}f=\mathrm{div}G$ in $\Om$ in the distributional sense. Then $f\in W^{1,2}_{loc}(\Om,\R^N)$ and for every $U\Subset\Om$ there exists $C>0$ independent of $f$, such that
$$
\|f\|_{1,2,U}\leq C(\|f\|_1+\|G\|_2).
$$
\end{Proposition}
\begin{proof}
Fix $U\Subset\Om'\Subset\Om$, with {}{$\partial\Om'\in C^2$}, and $\psi\in C^1_c(\Om,\R^N)$. Solve the system 
$$
\mathrm{L}\varphi=\psi\quad \mbox{ in }\Om', 
$$
in $W_0^{1,2}(\Om',\R^N)$. The existence of a weak solution $\varphi\in W^{1,2}_0(\Om',\R^N)$ is guaranteed by  \cite[Thm. 3.39 and Cor. 3.44]{GM}. From Schauder theory (see for instance \cite[Thm. 5.20]{GM}), we have that $\varphi\in C^2(\Om',\R^N)$. Moreover,  by $L^2$-regularity theory (see \cite[Thm. 4.14]{GM}), $\varphi\in W^{2,2}(\Om',\R^N)\cap W^{1,2}_0(\Om',\R^N)$ and
$$
\|\varphi\|_{2,2,\Om'}\leq C\|\psi\|_2
$$
for some constant $C=C(\Om',\lambda, \|A\|_{C^1(\Om')})$
Now notice that $\psi\in L^p(\Om,\R^N)$ for every $p\in[1,\infty]$. Thus, by Corollary \ref{interior 2,p est I}, for any $\Om''$ such that $U\Subset\Om''\Subset\Om'$ we can estimate 
$$
\|\varphi\|_{2,p,\Om''}\leq C\|\psi\|_{p} \quad\forall p\in[2,\infty)
$$
for some constant $C=C(\Om',\Om'',p,\lambda, \|A\|_{C^1(\Om')})$.
Moreover, if $p>n$, by Sobolev embeddings
$$
\|\varphi\|_{C^1(\Om'',\R^N)}\leq C\|\varphi\|_{2,p,\Om''}
$$
for a constant $C=C(n,p)$.
If $\chi\in C^\infty_c(\Om)$ with supp$\chi\subset\Om''$, up to approximation with $C^\infty_c$-functions (since $\varphi\in C^2$), we can test the equation $\mathrm{L}f=\mathrm{div}G$ with $\chi\varphi$, obtaining
$$
-\int_{\Om}G: \mathrm{D}(\chi\varphi)=\int_\Om f\cdot(\mathrm{div}(A\mathrm{D}\chi)\varphi)+\int_\Om \chi f\cdot \mathrm{div}(A\mathrm{D}\varphi)+2\int_\Om f\cdot(\mathrm{D}\chi A\mathrm{D}\varphi).
$$
Then, recalling that $\mathrm{L}\varphi=\mathrm{div}(A\mathrm{D}\varphi)=\psi$, we can estimate
$$
\left|\int_\Om\chi f\cdot\psi\right|\leq C\|\varphi\|_{C^1(\Om'',\R^N)}(\|f\|_1+\|G\|_1)\leq C\|\psi\|_p(\|f\|_1+\|G\|_1)
$$
for a constant $C=C(p,n,\Om',\Om'',\chi,\lambda, \|A\|_{C^1(\Om')})$.
So $\chi f$ can be regarded as an element of the dual of $L^p(\Om,\R^N)$. Since $\chi$ is arbitrary, we have proved that $f\in L^{p'}_{loc}(\Om'',\R^N)$ for $p>n$. Moreover, since $\Om''$ and $\Om'$ are arbitrary, we conclude that $f\in L^{p'}_{loc}(\Om,\R^N)$ for $ p>n$. By choosing $\chi=1$ on $U$, for all $p>n$ (i.e. all $p'<1^*$) there exists $C>0$ such that
\begin{align}\label{p' est}
\|f\|_{p',U}\leq C(\|f\|_1+\|G\|_1).
\end{align}
In order to gain one derivative, we repeat the same argument by solving instead
$$
\mathrm{L}\varphi=\mathrm{D}_\alpha\psi\quad \mbox{in } \Om'
$$
in $W_0^{1,2}(\Om',\R^N)$ for fixed $\psi\in C^2_c(\Om,\R^N)$, $U\Subset\Om'\Subset\Om$, and $\alpha\in\{1,\ldots,n\}$. As before, we have the existence of a solution $\varphi\in C^2(\Om',\R^N)$, and by Theorem \ref{Lp est}, for any $2\leq p<\infty$
$$
\|\varphi\|_{1,p,\Om'}\leq C\|\psi\|_{p}
$$
for a constant $C=C(p,\Om',\lambda, \|A\|_{C^1(\Om')})$. Now we test $\mathrm{L}f=g$ again with $\chi\varphi$, where $\chi\in C^\infty_c(\Om)$ and supp$\chi\subset\Om'$. For every $p>n$, using \eqref{p' est} applied to $\Om'$, we can estimate this time
\begin{align}\label{int est}
\left|\int_\Om\chi f\cdot \mathrm{D}_\alpha\psi\right|\leq C\|\varphi\|_{1,p,\Om'}(\|f\|_{p',\Om'}+\|G\|_{p'})\leq C\|\psi\|_p(\|f\|_{p',\Om'}+\|G\|_{2})\leq C\|\psi\|_p(\|f\|_{1}+\|G\|_{2})
\end{align}
for a constant $C=C(p,n,\Om',\chi,\lambda, \|A\|_{C^1(\Om')})$.
So, we conclude  that $f\in W^{1,p'}_{loc}(\Om,\R^N)$ for $p>n$, and by choosing $\chi=1$ on $U$, we infer
\begin{align}\label{1,p' est}
\forall p>n\quad(\mbox{i.e. }\forall p'<1^*) \quad\exists C>0:\quad\|f\|_{1,p',U}\leq C(\|f\|_1+\|G\|_{2}).
\end{align}
By the Sobolev imbedding theorems,
$$
\|f\|_{r,U}\leq C\|f\|_{1,p',U}\quad\forall r<1^{**}
$$
for a constant $C=C(n,p,U)>0$. So we can refine the estimate in \eqref{p' est} obtaining
$$
\forall p'<1^{**}\quad\exists C>0: \quad\|f\|_{p',U}\leq C(\|f\|_1+\|G\|_{2}).
$$
Using this estimate in \eqref{int est}, we can improve \eqref{1,p' est} to
$$
\forall p'<1^{**} \quad\exists C>0: \quad\|f\|_{1,p',U}\leq C(\|f\|_1+\|G\|_{2}).
$$
We conclude by iterating the argument until $1^{*\ldots*}\geq 2$, i.e. after $\lceil n/2\rceil $ steps.
\end{proof}

\begin{Corollary}{\bf (Interior $W^{1,p}$-estimates II)}\label{interior 1,p est II}
Let $\Om$ and $A$ be as in Proposition \ref{distrib to weak}. Assume that $f\in W^{1,2}_{loc}(\Om,\R^N)\cap L^1(\Om,\R^N)$ is a weak solution to $\mathrm{L}f=\mathrm{div}G$ in $\Om$, with $G\in L^p(\Om,\R^{N\times n})$, for some $p\in[2,\infty)$. Then $f\in W^{1,p}_{loc}(\Om,\R^N)$ and for every $U\Subset\Om$ there exists $C>0$ independent of $f$, such that
$$
\|f\|_{1,p,U}\leq C(\|f\|_1+\|G\|_p).
$$
\end{Corollary}
\begin{proof}
It is enough to utilise again the bootstrap argument in the proof of Proposition \ref{distrib to weak} until we reach the exponent $p$.
\end{proof}

\begin{Corollary}{\bf (Interior $W^{2,p}$-estimates II)}\label{interior 2,p est II}
Let $\Om$ and $A$ be as in Proposition \ref{distrib to weak}. Assume that $f\in W^{1,2}_{loc}(\Om,\R^N)\cap L^1(\Om,\R^N)$ is a weak solution to $\mathrm{L}f=g$ in $\Om$, with $g\in L^p(\Om,\R^{N})$, for some $p\in[2,\infty)$. Then $f\in W^{2,p}_{loc}(\Om,\R^N)$ and for every $U\Subset\Om$ there exists $C>0$ independent of $f$, such that
$$
\|f\|_{2,p,U}\leq C(\|f\|_1+\|g\|_p).
$$
\end{Corollary}
\begin{proof}
It is enough to derive the system satisfied by the discrete partial derivative $$\mathrm{D}_\gamma^h f:=\frac{1}{h}(f(\cdot+he_\gamma)-f(\cdot)),$$ for $|h|$ sufficiently small, and apply Corollary \ref{interior 1,p est II}.
\end{proof}

\subsection{Properties of convex functions}
In this second preliminary section, we prove some results concerning convex functions on $\R^N$. In what follows, we denote by $F_\xi$ and $F_{\xi\xi}$ respectively the gradient and the Hessian of a function $F:\R^N\to\R$.

\begin{Lemma}\label{convex lemma}
Let $F:\R^N\to\R$ be $C^1$ and convex, with $F(0)=0$. Then $$F(\xi)\leq F_\xi(\xi)\cdot\xi\quad\forall\xi\in\R^N.$$
\end{Lemma}
\begin{proof}
It is a direct consequence of \cite[Thm. 2.52(i)]{Dac}.

\begin{comment}
For $\xi\in\R^N$, set $\Phi(\xi):=F_\xi(\xi)\cdot\xi$.
Then $ \Phi_\xi(\xi)=F_{\xi\xi}(\xi)\xi+F_\xi(\xi)$. For any $\lambda\geq0$, since $F$ is convex, we get
$$
\frac{d}{d\lambda}[\Phi(\lambda\xi)]=\Phi_\xi(\lambda\xi)\cdot\xi=\lambda F_{\xi\xi}(\lambda\xi)\xi\cdot\xi+ F_\xi(\lambda\xi)\cdot\xi\geq F_\xi(\lambda\xi)\cdot\xi=\frac{d}{d\lambda}[F(\lambda\xi)].
$$
Therefore, by integrating on $(0,1)$, we conclude by using $F(0)=0$, as
$$
\Phi(\xi)=\int_0^1\frac{d}{d\lambda}[\Phi(\lambda\xi)]d\lambda\geq\int_0^1\frac{d}{d\lambda}[F(\lambda\xi)]d\lambda=F(\xi).
$$
\end{comment}
\end{proof}

We also need the next invertibility result. {}{Recall that $F:\R^N\to\R$ is strongly convex if $F-\frac{|\cdot|^2}{c}$ is convex, for some $c>0$. This is a stronger notion than strict convexity.}
\begin{Lemma}\label{invertibility}
Let $F:\R^N\to\R$ be $C^2$ and strongly convex, with $F\geq0$ and $F(0)=0$. Then the map $\Phi:\R^N\to\R^N$ defined by
{}{
\begin{equation}\nonumber 
\left\{
\begin{aligned}
&\Phi(\xi)=F(\xi)\frac{F_\xi(\xi)}{|F_\xi(\xi)|}\quad\forall\xi\in\R^N\setminus\{0\},\\
&\Phi(0)=0
\end{aligned}
\right.
\end{equation}
}
is a homeomorphism in $\R^N$.
\end{Lemma}
\begin{proof}
We divide the proof in two steps, showing first that $\Phi$ is a local diffeomorphism in $\R^N\setminus\{0\}$, and secondly that it is a global homeomorphism in $\R^N$.\\
{\textit{Step 1}}: We show that det$\Phi_\xi(\xi)>0$ away from the origin.\\
Let us compute the derivative of $\Phi$ at fixed $\xi\neq0$:
\begin{align*}
\Phi_\xi&=\frac{1}{|F_\xi|}\left(F_\xi\otimes F_\xi+FF_{\xi\xi}\left( I-\frac{F_\xi\otimes F_\xi}{|F_\xi|^2}\right)\right)\\
&=\frac{1}{|F_\xi|}\left(FF_{\xi\xi}+(|F_\xi|^2I-FF_{\xi\xi})\frac{F_\xi\otimes F_\xi}{|F_\xi|^2}\right),\quad\mbox{in }\R^N\setminus\{0\}.
\end{align*}
Notice that it is well defined and continuous in $\R^N\setminus\{0\}$, since our assumptions imply that both $F$ and $F_\xi$ vanish only at $\xi=0$.
Now we recall the following formula concerning the determinant of rank 1-perturbation of invertible matrices:
$$
\mathrm{det}(A+a\otimes b)=(\mathrm{det}A)(1+(A^{-1}a)\cdot b)\quad\forall A\in GL_N(\R)\quad\forall a,b\in\R^N.
$$
Applying it to $A=FF_{\xi\xi}$, $a=(|F_\xi|^2I-FF_{\xi\xi})\frac{F_\xi}{|F_\xi|}$, $b=\frac{F_\xi}{|F_\xi|}$, we get
\begin{align*}
\mathrm{det}\Phi_\xi&=\frac{1}{|F_\xi|^N}\mathrm{det}(FF_{\xi\xi})\left(1+(FF_{\xi\xi})^{-1}\left((|F_\xi|^2I-FF_{\xi\xi})\frac{F_\xi}{|F_\xi|}\right)\cdot\frac{F_\xi}{|F_\xi|}\right)\\
&=\frac{F^{N-1}}{|F_\xi|^N}\mathrm{det}(F_{\xi\xi})\left((F_{\xi\xi})^{-1}F_\xi\cdot F_\xi\right)\\
&>0\quad\mbox{in }\R^N\setminus\{0\},
\end{align*}
where we used that det$(F_{\xi\xi})>0$ and that $(F_{\xi\xi})^{-1}$ is positive definite, by {}{strong convexity} of $F$.\\

\noindent
{\textit {Step 2}}: We show that $\Phi$ is a global homeomorphism in $\R^N$.\\
Notice that $\Phi$ is continuous on $\R^N$ and $\Phi(0)=0$. Since $F$ is strictly convex, positive and null at $0$, we have that $\Phi^{-1}(\{0\})=\{0\}$, which implies $\Phi(\R^N\setminus\{0\})\subset\R^N\setminus\{0\}$.
 Let us prove that the restriction  $\Phi|_{\R^N\setminus\{0\}}:\R^N\setminus\{0\}\to\R^N\setminus\{0\}$ is surjective, by showing that $\Phi(\R^N\setminus\{0\})$ is open and closed in $\R^N\setminus\{0\}$.\\
Since $\Phi$ is a local diffeomorphism in $\R^N\setminus\{0\}$, it is immediate to prove that $\Phi(\R^N\setminus\{0\})$ is open. To show the closedness, consider a sequence ${(\eta_k)}_k\subset\Phi(\R^N\setminus\{0\})$ converging to some $\eta_\infty\in\R^N\setminus\{0\}$ and let ${(\xi_k)}_k\in\R^N\setminus\{0\}$ be such that $\Phi(\xi_k)=\eta_k$. Our assumptions on $F$ implies 
\begin{align}\label{phi to infty}
\lim_{|\xi|\to\infty}F(\xi)=\lim_{|\xi|\to\infty}|\Phi(\xi)|=\infty.
\end{align}
So, ${(\xi_k)}_k$ is necessarily bounded and there exists a subsequence ${(\xi_{k_j})}_j\subset{(\xi_k)}_k$ such that $\xi_{k_j}\to\xi_\infty$, for some $\xi_\infty\in\R^N$, as $j\to\infty$. Therefore, we conclude
\begin{align*}
\eta_\infty=\lim_{j\to\infty}\eta_{k_j}=\lim_{j\to\infty}\Phi(\xi_{k_j})=\Phi(\xi_\infty),
\end{align*}
that also ensures $\xi_\infty\neq0$.\\
%In particular, we have proved that $\Phi(\R^N)=\R^N$, i.e. $\Phi$ is surjective.  
Now we claim that any $\eta\in\R^N\setminus\{0\}$ has finitely many preimages. Indeed, if for the sake of contradiction we assume there exists a sequence of distinct points ${(\xi_k)}_k\subset\R^N$ such that $\Phi(\xi_k)=\eta$, then ${(\xi_k)}_k$ is bounded, from \eqref{phi to infty}. So, there exists a subsequence ${(\xi_{k_j})}_j\subset{(\xi_k)}_k$ such that $\xi_{k_j}\to\xi_\infty$, for some $\xi_\infty\in\R^N\setminus\{0\}$, as $j\to\infty$. Then any neighbourhood of $\xi_\infty$ contains infinitely many preimages of $\eta$, contradicting the local invertibility of $\Phi$ in $\R^N\setminus\{0\}$. Therefore, we have proved that $\Phi$ is a topological covering of $\R^N\setminus\{0\}$. For $N\geq3$, we infer that $\Phi$ is a homeomorphism in $\R^N\setminus\{0\}$, since it is a universal covering of a simply connected space (see \cite[Prop. 13.30]{Ma}). Finally, since $\Phi$ is continuous at $0$ and $\Phi^{-1}(0)=\{0\}$, it is not difficult to see that $\Phi$ extends to a global homeomorphism in $\R^N$, for any $N\geq3$.\\
It remains to discuss the case $N\leq 2$. The one dimensional case is immediate, so let us assume $N=2$. Extend $F$ to $\R^3$ by defining the function
$$
\tilde F(\xi_1,\xi_2,\xi_3)=F(\xi_1,\xi_2)+\xi_3^2.
$$
Then $\tilde F:\R^3\to\R$ is $C^2$ and {}{strongly convex}, with $\tilde F\geq0$ and $\tilde F(0)=0$. Thanks to the previuos argument, the map
$$
\tilde \Phi(\xi)=\tilde F(\xi)\frac{\tilde F_\xi(\xi)}{|\tilde F_\xi(\xi)|}\quad\forall\xi\in\R^3
$$
is a homeomorphism in $\R^3$. Since $\tilde \Phi|_{{\R^2}\times\{0\}}=\Phi$, we obtain the conclusion.
\end{proof}

\section{Proof of the main results}\label{sec:proofs}
{}{First, we prove Theorem \ref{existence thm}, concerning existence of minimisers. Our arguments follow similar lines to those in \cite{KM}, appropriately extended to the vectorial case and written in the formalism of Gamma convergence.
Then, we divide the proof of Theorem \ref{main thm} into two parts, about PDE derivation and uniqueness respectively. In the first one, we follow the strategy in \cite{KM}, which adapts to the vectorial setting up to some technical issues, while in the last part our arguments are significantly different.}

\subsection{Proof of Theorem \ref{existence thm}}\label{subs:existence}
Let $p\geq1$ and recall that $u_0\in\mathcal{W}^{2,\infty}(\Om,\R^N)$. We begin by considering the functional
$$
E_p(u):=\left(\fint_{\Om}F(x,\mathrm{L} u(x))^pdx\right)^{\frac{1}{p}},\quad u\in W_{u_0}^{2,2p}(\Om,\R^N).
$$
{}{
By Lemma \ref{convex lemma} and assumption \eqref{ass2}, we have $F(x,\xi)\leq c|\xi|^2$ for every $\xi\in\R^N$ and almost every $x\in\Om$. Therefore,
$$
E_p(u_0)\leq E_\infty(u_0)\leq c\,\|\mathrm Lu_0\|^2_\infty<\infty.
$$
In particular, this guarantees that $e_\infty<\infty$.}\\
Assume that $E_p(u)<\infty$ for some $u\in W_{u_0}^{2,2p}(\Om,\R^N)$. {}{From \eqref{ass1}, we infer that $\mathrm{L}u\in L^{2p}(\Om,\R^N)$}. By applying Corollary \ref{CZ} and \eqref{ass1}, we have 
$$
\|u\|_{2,2p}\leq C\left(\|\mathrm{L}u\|_{2p}+\|u_0\|_{2,2p}\right)\leq C\left(c^\frac12|\Om|^\frac{1}{2p}E_p(u)^\frac12+\|u_0\|_{2,2p}\right)
$$
for a constant $C=C(\Om,p,\lambda,A)$.
Thus, we deduce that $E_p$ is coercive in $W_{u_0}^{2,2p}(\Om,\R^N)$. Moreover, since $F$ is convex with respect to $\xi$, by Direct Methods we obtain the existence of a minimiser $u_p\in W_{u_0}^{2,2p}(\Om,\R^N)$ of $E_p$. \\
Fix $k\in\N$. Extend $E_p$ and $E_\infty$ to $W^{2,k}_{u_0}(\Om,\R^N)$ as
\begin{equation}\nonumber
E_p(u)=\left\{
\begin{aligned}
&\left(\fint_{\Om}F(x,\mathrm{L} u(x))^p\right)^{\frac{1}{p}}\!\!\!\!\!&,\quad&u\in W_{u_0}^{2,2p},\\
&+\infty&,\quad&u\in W_{u_0}^{2,k}\setminus W_{u_0}^{2,2p},
\end{aligned}
\right.\,\,\,
E_\infty(u)=\left\{
\begin{aligned}
&\|F(\cdot,\mathrm{L}u)\|_\infty \!\!\!\!\!&,\quad& u\in \mathcal{W}_{u_0}^{2,\infty},\\
&+\infty &,\quad&  u\in W_{u_0}^{2,k}\setminus \mathcal{W}_{u_0}^{2,\infty}.
\end{aligned}
\right.
\end{equation}
Let us show that $(E_p)_p$ is equi-coercive in $W^{2,k}_{u_0}$ with respect to $p$. Assume $E_p(u)\leq M$ for all $p\in[k/2,\infty)$. By \eqref{ass1} and the H\"older inequality, we have
\begin{align}\label{equi coerciv bound}
c^{-\frac12}|\Om|^{-\frac{1}{2k}}\|\mathrm{L}u\|_k\leq E_k(u)^\frac12\leq  E_{2p}(u)^\frac12\leq M^\frac12.
\end{align}
Hence, by Corollary \ref{CZ}
$$
\|u\|_{2,k}\leq C\left(\|\mathrm{L}u\|_k+\|u_0\|_{2,k}\right)\leq C\left(c^\frac12 |\Om|^\frac1k M^\frac12+\|u_0\|_{2,k}\right),
$$
for a constant $C=C(\Om,k,\lambda,A)$. This gives the equi-coercivity of $(E_p)_p$ in $W^{2,k}_{u_0}$. \\
Moreover, $(E_p)_p$ is a monotone increasing sequence of weakly lower semicontinuous functionals
 on $W^{2,k}_{u_0}$, thus its $\Gamma$-limit is given by (see \cite[Remark 2.12]{Br})
$$
\Gamma-\lim_p E_p=\lim_pE_p=E_\infty.
$$
By the Fundamental Theorem of Gamma Convergence \cite[Thm. 2.10]{Br}, if $(u_p)_p$ is a sequence of minimisers of $E_p$, up to passing to a subsequence, we have $u_p\rightharpoonup u_\infty$ weakly $W_{u_0}^{2,k}$ where $u_\infty$ is a minimiser of $E_\infty$ on $W_{u_0}^{2,k}$. By definition of $E_\infty$, we necessarily have $u_\infty\in\mathcal{W}^{2,\infty}_{u_0}(\Om,\R^N)$.
Moreover, by the same Theorem
\begin{align}\label{limit ep}
E_p(u_p)\to E_\infty({u_\infty}) \quad\mbox{as }p\to\infty.
\end{align}

\subsection{Proof of Theorem \ref{main thm}: {PDE derivation}}\label{subs: PDE derivation}

Now we work towards the proof of \eqref{PDE u infty}.
Let us begin by writing down the system of Euler-Lagrange equations satisfied by $u_p$, using the variational nature of the operator $L$ and the symmetry of $A$: 
\begin{align}\label{EL for Ep}
\mathrm{L} f_p=0,\quad \mbox{ in }\mathcal{D}'(\Om,\R^N),
\end{align}
where, for $e_p:=E_p(u_p)$, $f_p:\Om\to\R^N$ is defined by
\begin{align}\label{fp}
f_p:=e_p^{1-p}F(\cdot,\mathrm{L} u_p)^{p-1} F_\xi(\cdot,\mathrm{L} u_p).
\end{align}
Since the only interesting case is $e_\infty>0$, we may assume that $e_p\geq\alpha>0$ (for $p$ large enough). In fact, if $e_\infty=0$, then $F(\cdot,\mathrm{L}u_\infty)=0$ and system \eqref{PDE u infty} reduces to $\mathrm{L}u_\infty=0$.
For this reason, $f_p$ is a well defined measurable map.\\
Let us prove that that $(f_p)_p$ is bounded in $L^1(\Om,\R^N)$:
\begin{equation}\label{fp bdd}
\begin{aligned}
\fint_\Om|f_p(x)|dx&=e_p^{1-p}\fint_\Om F(x,\mathrm{L} u_p(x))^{p-1}|F_\xi(x,\mathrm{L} u_p(x))|dx\\
&\leq ce_p^{1-p}\fint_\Om F(x,\mathrm{L} u_p(x))^{p-1}|\mathrm{L}u_p(x)|dx\\
&\leq ce_p^{1-p}\left(\fint_\Om F(x,\mathrm{L} u_p(x))^pdx\right)^\frac{p-1}{p}\left(\fint_\Om |\mathrm{L} u_p(x)|^pdx\right)^\frac{1}{p}\\
&=c\left(\fint_\Om |\mathrm{L} u_p(x)|^pdx\right)^\frac{1}{p}\\
&\leq c\left(\fint_\Om[c F(x,\mathrm{L} u_p(x))]^\frac{p}{2}dx\right)^\frac{1}{p}\\
&\leq c{c}^\frac{1}{2}\left(\fint_\Om F(x,\mathrm{L} u_p(x))^pdx\right)^\frac{1}{2p}=c^\frac{3}{2}{e_p}^\frac{1}{2}\leq c^\frac{3}{2}{e_\infty}^\frac{1}{2}.
\end{aligned}
\end{equation}
In particular, thanks to Proposition \ref{distrib to weak}, $f_p\in W^{1,2}_{loc}(\Om,\R^N)\cap L^1(\Om,\R^N)$ is a weak solution of \eqref{EL for Ep}. Thus, by Corollary \ref{interior 2,p est II}, for every $q<\infty$,
$$
\|f_p\|_{2,q,U}\leq C\|f_p\|_1\leq C,\quad\forall U\Subset\Om,
$$
for some constant $C=C(q,U,\Om,c,e_\infty)>0$. Hence, by Morrey's theorem, we deduce that ${(f_p)}_p$ is bounded in the local topology of $C^1$, and so there exists $f_\infty$ and a subsequence ${(f_{p_\ell})}_\ell$ such that $f_{p_\ell}\to f_\infty$ locally uniformly as $\ell\to\infty$. In particular, $\mathrm{L}f_\infty=0$ in the sense of distributions and $f_\infty\in L^1(\Om,\R^N)$ as well. \\
The next lemma is central in our argument.
\begin{Lemma}{\bf (Non-triviality of $f_\infty$)}\label{opt lemma}
The map $f_\infty$ constructed above is such that $f_\infty\not\equiv0$.
\end{Lemma}
\begin{proof}
First we show that $f_p\stackrel{*}{\rightharpoonup}f_\infty$ in $\mathcal{M}(\overline\Om,\R^N)$, as $p\to\infty$ along a subsequence.\\
 By approximating with $C_c^\infty$-maps, thanks to the regularity assumptions on $A$, for every $ p\in(1,\infty)$ we have 
\begin{align}\label{strong A-harm}
\mathrm{L}f_p=\mathrm{L}f_\infty=0 \quad\mbox{ in }\left(W^{2,\infty}_0(\Om,\R^N)\right)^*.
\end{align}
From the boundedness in $L^1$, there exists a measure $\mu\in\mathcal{M}(\overline\Om,\R^N)$ such that $f_p\mathscr{L}^n\mres\Om\stackrel{*}{\rightharpoonup}\mu$ in $\mathcal{M}(\overline\Om,\R^N)$. Since $f_p\rightharpoonup f_\infty$ in $W^{2,q}_{loc}(\Om,\R^N)$, we necessarily have $\mu\mres\Om=f_\infty\mathscr{L}^n\mres\Om$. So we can write $$\mu=f_\infty\mathscr{L}^n\mres\Om+\mu\mres\partial\Om.$$
Now we claim that $\mu\mres\partial\Om=0$.
Let us denote $d:=\mathrm{dist}(\cdot,\partial\Om)$ and $\Om_r:=\{x\in\Om:d(x)\leq r\}$. Then, by the regularity of $\partial\Om$, we have $d\in C^2(\overline\Om_r)$ for some $r>0$ sufficiently small. Extend it to a function $d\in C^2(\overline\Om)$ and fix a map $g\in C^2(\overline\Om,\R^N)$. Set $\psi:=d^2g\in  C^2(\overline\Om,\R^N)\cap W^{2,\infty}_0(\Om,\R^N)$ and test \eqref{strong A-harm} with $\psi$:
\begin{align}\label{test psi}
0=\int_\Om(f_p-f_\infty)\cdot \mathrm{L}\psi =\int_\Om(f_p-f_\infty)\cdot[(\mathrm{div}A) \mathrm{D}\psi+A\mathrm{D}^2\psi].
\end{align}
By a computation,
\begin{align*}
&\mathrm{D}\psi=2d\,g\otimes\mathrm{D}d+d^2\mathrm{D}g ,\\
&\mathrm{D}^2\psi=2g\otimes\mathrm{D}d\otimes \mathrm{D}d+2d\, g\otimes\mathrm{D}^2d+4d\mathrm{D}g\otimes \mathrm{D}d+d^2\mathrm{D}^2g.
\end{align*}
Thus,
$$
\mathrm{D}^2\psi\mres\partial\Om=2(g\otimes\mathrm{D}d\otimes \mathrm{D}d)\mres\partial\Om=2\, g\otimes\nu \otimes\nu\mres\partial\Om,
$$
where $\nu$ is the outwards unit normal to the boundary.
Since $(\mathrm{div}A) \mathrm{D}\psi\in C_0(\Om,\R^N)$ and \\$(f_p-f_\infty)\mathscr{L}^n\mres\Om\stackrel{*}{\rightharpoonup}\mu\mres\partial\Om$ in $\mathcal{M}(\overline\Om,\R^N)$, we can pass to the limit in \eqref{test psi} to obtain
$$
0=2\int_{\partial\Om}A:g \otimes\nu \otimes\nu\cdot d\mu.
$$
By the ellipticity of $A$, the matrix $M:=A:\nu\otimes\nu$ is positive definite, so we conclude that
$$
\int_{\partial\Om}Mg\cdot d\mu=0,\quad\forall g\in C^2(\overline\Om,\R^N),
$$
which gives $\mu\mres\partial\Om=0$.\\
By Lemma \ref{convex lemma}, we estimate
\begin{align*}
\int_\Om f_p\cdot \mathrm{L} u_p&=e_p^{1-p}\int_\Om F(x,\mathrm{L} u_p(x))^{p-1}F_\xi(x,\mathrm{L} u_p(x))\cdot \mathrm{L} u_p(x)dx\\
&\geq e_p^{1-p}\int_\Om F(x,\mathrm{L} u_p(x))^pdx\\
&=e_p^{1-p}e_p^p|\Om|={|\Om|}e_p\geq|\Om|\alpha,
\end{align*}
where in the last step, we recall we are under the assumption $e_\infty>0$.\\
Thanks to Whitney's extension theorem, there exists $u_*\in C^2(\R^n,\R^N)$ such that $u_*=u_0$ and $\mathrm{D}u_*=\mathrm{D}u_0$ on $\partial\Om$. Now we observe that \eqref{EL for Ep} clearly holds also in $\left(W^{2,2p}_0(\Om,\R^N)\right)^*$, therefore by testing with $u_p-u_*\in W^{2,2p}_0(\Om,\R^N)$, we obtain
$$
|\Om|\alpha\leq\int_\Om f_p\cdot \mathrm{L}u_p=\int_\Om f_p\cdot \mathrm{L}u_*.
$$ 
So, we can pass to the limit on the right hand side, observing that $\mathrm{L}u_*\in C(\overline\Om,\R^N)$, and conclude
$$
\int_\Om f_\infty\cdot \mathrm{L}u_*\geq|\Om|\alpha.
$$
In particular, this ensures that $f_\infty\not\equiv0$.
\end{proof}
Thanks to Lemma \ref{opt lemma}, we are now able to derive the PDE satisfied by $u_\infty$. \\
Set $\Gamma_\infty:=f_\infty^{-1}(0)\subsetneq\Om$. From \eqref{fp}, we obtain the equation
\begin{align}\label{inverting expr 1}
|f_p(x)|^\frac{2}{2p-1}\frac{f_p(x)}{|f_p(x)|}=e_p^{-1+\frac{1}{2p-1}}F(x,\mathrm{L} u_p(x))\left(\frac{|F_{\xi}(x,\mathrm{L} u_p(x))|}{\sqrt{F}(x,\mathrm{L} u_p(x))}\right)^\frac{2}{2p-1}\frac{F_\xi(x,\mathrm{L} u_p(x))}{|F_\xi(x,\mathrm{L} u_p(x))|},\qquad\mbox{a.e. }x\in \Om,
\end{align}
which we rewrite as
\begin{align}\label{inverting expr 2}
e_p^{1-\frac{1}{2p-1}}\left(\frac{|F_{\xi}(x,\mathrm{L} u_p(x))|}{\sqrt{F}(x,\mathrm{L} u_p(x))}\right)^{-\frac{2}{2p-1}}|f_p(x)|^\frac{2}{2p-1}\frac{f_p(x)}{|f_p(x)|}=\Phi(x,\mathrm{L} u_p(x)),\qquad\mbox{a.e. }x\in \Om,
\end{align}
where we have symbolised
$$\Phi(x,\xi):=F(x,\xi)\frac{F_\xi(x,\xi)}{|F_\xi(x,\xi)|},\qquad\mbox{ a.e. }  x\in\Om,\quad\forall\xi\in\R^N.$$
Notice that equations \eqref{inverting expr 1} and \eqref{inverting expr 2} are well defined. Indeed, from \eqref{ass1},\eqref{ass2}, and Lemma \ref{convex lemma}, it is easy to see that $|F_\xi|/\sqrt{F}$ is bounded from above and below. Moreover, by the assumptions on $F$, we have
$f_p=0$ on a set of positive measure $E\subset \Om$ if and only if $\mathrm Lu_p=0$ on the same set. So, points in $E$ are not singular for both sides of \eqref{inverting expr 1} and \eqref{inverting expr 2}, naturally vanishing on $E$.\\
{}{Thanks to the strong convexity assumption \eqref{ass0},} we can apply Lemma \ref{invertibility}, obtaining that for almost every $x\in\Om$, the map $\Phi_x:=\Phi(x,\cdot):\R^N\to\R^N$ is a homeomorphism. So, applying the inverse map $\Phi_x^{-1}$, we rewrite the previous equation as
$$
\mathrm{L} u_p(x)=\Phi_x^{-1}\left(e_p^{1-\frac{1}{2p-1}}\left(\frac{|F_{\xi}(x,\mathrm{L} u_p(x))|}{\sqrt{F}(x,\mathrm{L} u_p(x))}\right)^{-\frac{2}{2p-1}}|f_p(x)|^\frac{2}{2p-1}\frac{f_p(x)}{|f_p(x)|}\right),\qquad\mbox{a.e. }x\in \Om.
$$
 Recall that $f_{p_\ell}\to f_\infty$ uniformly on every compact $K\subset\Om\setminus\Gamma_\infty$ and, by \eqref{limit ep}, $e_{p_\ell}\to e_\infty$ as $\ell\to\infty$. So we can pass to the limit locally uniformly in the right hand side along the subsequence ${(p_\ell)}_\ell$. On the other hand, since we already know that $\mathrm{L}u_p\rightharpoonup \mathrm{L}u_\infty$ weakly in $L^q$ for every $q\geq1$, we obtain
$$
\mathrm{L}u_\infty(x)=\Phi_x^{-1}\left(e_\infty\frac{f_\infty(x)}{|f_\infty(x)|}\right),\qquad\mbox{a.e. }x\in K.
$$
By inverting again, we get
$$
\Phi(x,\mathrm Lu_\infty)=F(x,\mathrm Lu_\infty)\frac{F_\xi(x,\mathrm Lu_\infty)}{|F_\xi(x,\mathrm Lu_\infty)|}=e_\infty\frac{f_\infty(x)}{|f_\infty(x)|},\qquad\mbox{a.e. }x\in K.
$$
Since $\mathrm{L}f_\infty=0$, the unique continuation property of L yields $|\Gamma_\infty|=0$. As a consequence, we obtain \eqref{PDE u infty}.

\subsection{Proof of Theorem \ref{main thm}: {Uniqueness}}\label{subs: uniqueness}
In the previous sections, we proved the existence of a minimiser $u_\infty\in\mathcal{W}^{2,\infty}_{u_0}(\Om,\R^N)$ of $E_\infty$ and that it satisfies the PDE system \eqref{PDE u infty}. As a first step to provide uniqueness of minimisers of $E_\infty$ in $\mathcal{W}^{2,\infty}_{u_0}(\Om,\R^N)$, we show that every minimiser of $E_\infty$ satisfies a PDE system of the same structure as \eqref{PDE u infty}. Some of the arguments are similar to Section \ref{subs:existence}, but we will present the complete proof for sake of clarity. More precisely, we start with proving the following result.
\begin{Lemma}{\bf (Necessity of the PDE system)}\label{necessity of PDE}
Let $F, A$, and $u_0$ be as in Theorem \ref{main thm}. Suppose that $u\in\mathcal{W}^{2,\infty}_{u_0}(\Om,\R^N)$ is a minimiser of $E_\infty$. Then there exists $f\in L^1(\Om,\R^N)$ such that $f\in W^{2,q}_{loc}(\Om,\R^N)$ for all $q<\infty$ and
\begin{equation}\label{PDE u}
\left\{
\begin{aligned}
&F(x,\mathrm{L} u(x))\frac{F_\xi(x,\mathrm{L} u(x))}{|F_\xi(x,\mathrm{L} u(x))|}=e_\infty\frac{f(x)}{|f(x)|},\quad \mbox{a.e. } x\in\Om,\\
&\mathrm{L} f(x)=0, \quad \mbox{a.e. } x\in\Om.
\end{aligned}
\right.
\end{equation}
\end{Lemma}
\begin{proof}
Let $u\in\mathcal{W}^{2,\infty}_{u_0}(\Om,\R^N)$ be a minimiser of $E_\infty$. For any $p\geq1$, we consider the auxiliary functional
$$
A_p(v):=E_p(v)+\frac{1}{2}\fint_\Om|v-u|^2,\,\qquad v\in W^{2,2p}_{u_0}(\Om,\R^N).
$$
Using the Direct Method, it is not difficult to see that $A_p$ attains a minimum point. Let $v_p\in W^{2,2p}_{u_0}(\Om,\R^N)$ be a minimiser of $A_p$, then, recalling that $u_p$ is a minimiser of $E_p$ in the same space, we have
$$
e_p=E_p(u_p)\leq E_p(v_p)\leq A_p(v_p)\leq A_p(u)=E_p(u)\leq E_\infty(u)=e_\infty.
$$
Since $e_p\to e_\infty$, we have $A_p(v_p)\to e_\infty$ as $p\to\infty$. Now, if $(v_p)_p$ is a sequence of minimisers of $A_p$, by Calderon-Zygmund estimates, we can prove as in Section \ref{subs:existence} that $(v_p)_p$ bounded in $W^{2,q}_{u_0}(\Om,\R^N)$ for any fixed $q\geq1$. So, $v_{p_\ell}\rightharpoonup v_\infty$ weakly $ W^{2,q}_{u_0}$ for a subsequence $(p_\ell)_\ell$ and some $v_\infty\in W^{2,q}_{u_0}(\Om,\R^N)$.\\
We claim that $v_\infty=u$. Indeed, passing to the limit as $\ell\to\infty$ in the following inequality
$$
e_{p_\ell}+\frac{1}{2}\fint_\Om|v_{p_\ell}-u|^2\,\leq A_{p_\ell}(v_{p_\ell}),
$$
we get
$$
e_\infty+\frac{1}{2}\fint_\Om|v_\infty-u|^2\,\leq e_\infty,
$$
which necessarily implies $v_\infty=u$.\\
Now, for simplicity let us denote $v_p=v_{p_\ell}$ and write down the Euler-Lagrange equation satisfied by $v_p$:
\begin{align}\label{EL for vp}
\mathrm{L} g_p+v_p-u=0\qquad\mbox{in }\mathcal{D}'(\Om,\R^N),
\end{align}
 where, for $a_p:=E_p(v_p)$, $g_p$ is defined by $g_p:=a_p^{1-p}F(\cdot,\mathrm{L} v_p)^{p-1}F_\xi(\cdot,\mathrm{L} v_p)$. We can prove that $g_p$ is well defined and bounded in $L^1$ as showed for $f_p$ in {}{Section \ref{subs: PDE derivation}}.
%Notice that $$\Delta g_p\rightharpoonup 0\quad \mbox{ weakly in }W^{2,q}_0(\Om,\R^N)\quad \forall q\geq1.$$
In particular, thanks to Proposition \ref{distrib to weak}, $g_p\in W^{1,2}_{loc}(\Om,\R^N)\cap L^1(\Om,\R^N)$ is a weak solution of \eqref{EL for vp}. Thus, by Corollary \ref{interior 2,p est II}, for every $q<\infty$,
$$
\|g_p\|_{2,q,U}\leq C\left(\|g_p\|_1+\|v_p-u\|_{q}\right)\leq C,\quad\forall U\Subset\Om,
$$
for some constant $C=C(q,U,\Om,c,e_\infty)>0$.
By Sobolev embeddings, up to passing to a (not relabeled) subsequence, $g_p\to g_\infty$ locally uniformly as $p\to\infty$, for some $g_\infty\in L^1(\Om,\R^N)\cap W^{2,q}_{loc}(\Om,\R^N) $ such that $\mathrm{L}g_\infty=0$. We can prove that $g_\infty\not\equiv0$ as in Lemma \ref{opt lemma}, since the same argument holds even if $\mathrm{L} g_p\rightharpoonup0$ in the sense of distributions only. Finally, since $a_p\to e_\infty$ as $p\to\infty$, we can pass to the limit in the expression defining $g_p$ exactly as we did for $f_p$, obtaining \eqref{PDE u} with $f=g_\infty$.
\end{proof}

We conclude this section with the proof of uniqueness of the solution to the minimum problem associated to $E_\infty$. Let $u_1$ and $u_2$ be two minimisers in $\mathcal{W}^{2,\infty}_{u_0}(\Om,\R^N)$, then, by \eqref{PDE u}, we have
$$
F(\cdot,\mathrm{L} u_1)=F(\cdot,\mathrm{L} u_2)=e_\infty,\quad\mbox{a.e. in } \Om.
$$
By convexity of $F(x,\cdot)$ and linearity of L, for a.e. $x\in\Om$ we have
\begin{equation}\label{average}
\begin{aligned}
F\left(x,\mathrm{L}\left(\frac{u_1(x)+u_2(x)}{2}\right)\right)&=F\left(x,\frac{\mathrm{L} u_1(x)}{2}+\frac{\mathrm{L} u_2(x)}{2}\right)\\
&\leq\frac{1}{2}F(x,\mathrm{L} u_1(x))+\frac{1}{2}F(x,\mathrm{L}u_2(x))\\
&=e_\infty.
\end{aligned}
\end{equation}
The average $\frac{u_1+u_2}{2}$ still belongs to $\mathcal{W}^{2,\infty}_{u_0}(\Om,\R^N)$, then necessarily
$$
F\left(\cdot\,,\mathrm{L}\left(\frac{u_1+u_2}{2}\right)\right)=e_\infty,\quad\mbox{a.e. in } \Om.
$$
{}{But since $F(x,\cdot)$ is strongly convex, it is also strictly convex}, so the only possibility to have equality in \eqref{average} is that $\mathrm{L} u_1=\mathrm{L} u_2$, giving $u_1=u_2$ by uniqueness results for elliptic systems with Dirichlet boundary conditions (see \cite[Theorems 3.39 and 3.46]{GM}). 
\\

\section{On the unique continuation property}\label{ucp}
In this last section, we identify classes of operators satisfying the unique continuation property (see Definition \ref{def ucp}).
First, it is worth mentioning that a slightly different notion that often appears in the literature is the \textit{strong unique continuation property}, namely the only solution to $\mathrm{L}u=0$ with a zero of infinite order is the trivial one. We recall that $x_0\in \Om$ is a zero of infinite order for $u$ if $$\int_{{B_r}(x_0)}|u|^2dx=O(r^k)\quad\mbox{ as }r\to0\quad\forall k\in\N.$$ There is also the definition of \textit{weak unique continuation property} (the only solution to $\mathrm{L}u=0$ vanishing on an open set is the trivial one), but we are not interested in that one. For a proof of the fact that the strong definition implies the one in Definition \ref{def ucp}, we refer to \cite{FG}. Their proof actually applies to elliptic equations, but it can be identically readapted for systems, since it is basically a consequence of Caccioppoli inequalities \cite[Sec. 4.2]{GM}.
However, differently from the scalar case, where the zero nodal set is even $(n-1)$-dimensional \cite{HS}, we have no theory available that guarantees sufficient conditions for (strong) unique continuation property for elliptic systems. For this reason, we restrict ourself to list some examples of interest, although one of them is deeply connected to the scalar case. In what follows, we assume that $u\in L^1(\Om,\R^N)$ is a distributional solution of $\mathrm{L}u=\mathrm{div}(A\mathrm{D}u)=0$ with $A$ satisfying either (H1) or (H2). Notice that by Proposition \ref{distrib to weak}, $u$ is actually a weak solution. Without loss of generality, we assume that $u$ is not constant. We consider the following cases:
\begin{itemize}
\item $A$ has analytic coefficients.\\
In this case, it is well known  that solutions to $\mathrm Lu=0$ are analytic \cite{OR}. In particular, if $u$ is not constant, its nodal sets are negligible.

\item $A$ has the form 
$$A=\sum_{h=1}^Ne_h\otimes e_h \otimes B_h$$ 
with $\{e_h\}_{h=1,\ldots,N}$ orthonormal base of $\R^N$ and $B_h\in C^1(\overline\Om,\R^{n\times n})$ elliptic for every $h=1,\ldots,N$.\\
We have
$$
(\mathrm{L}u)_i=\mathrm{D}_\alpha\left(\left[\sum_{h=1}^Ne_{h_i}e_{h_j}B_h^{\alpha\beta}\right]\mathrm{D}_\beta u^j\right)=0\quad\forall i=1,\ldots,N.
$$
Fix $k\in\{1,\ldots,N\}$ and take the scalar product of $Lu$ with $e_k$:
\begin{align*}
0&=e_k^i(\mathrm{L}u)_i\\
&=\mathrm{D}_\alpha\left(\left[\sum_{h=1}^Ne_k^ie_{h_i}e_{h_j}B_h^{\alpha\beta}\right]\mathrm{D}_\beta u^j\right)\\
&=\mathrm{D}_\alpha\left(e_{k_j}B_k^{\alpha\beta}\mathrm{D}_\beta u^j\right)\\
&=\mathrm{D}_\alpha\left(B_k^{\alpha\beta}\mathrm{D}_\beta(e_{k_j} u^j)\right)
\end{align*}
Set $u^{(k)}=e_k\cdot u$, namely the $k$-{th} component of $u$ w.r.t. the base $\{e_h\}$. Then, for each $k=1,\ldots,N$, $u^{(k)}$ solves an elliptic equation in divergence form. Since we assumed that $u$ is not constant, there is at least one $k$ such that $u^{(k)}$ is not constant. So, by applying the scalar theory \cite{HS} to $u^{(k)}$, since
$
\{|u|=0\}\subset\{u^{(k)}=0\}
$, the unique continuation must hold.

\end{itemize} 

\textsc{Acknowledgements:} S.C and R.M. acknowledge partial financial support through the EPSRC grant EP/X017206/1. S.C and N.K. acknowledge partial financial support through the EPSRC grant EP/X017109/1. 
S.C. is a member of Gruppo Nazionale per
l’Analisi Matematica, la Probabilità e le loro Applicazioni (GNAMPA) of the Istituto Nazionale
di Alta Matematica (INdAM) of Italy.\\

\textsc{Data availability}: Our manuscript has no associated data.\\

\textsc{Conflict of interest}: The authors have no Conflict of interest to declare for this article.

%%%%%%%%%%%%%        BIBLIOGRAFIA        %%%%%%%%%%%%%%%%%%%%%%%%%%%%

\vspace{3mm}
\textsc{Simone Carano}\\
Department of Mathematical Sciences, University of Bath\\
Bath BA2 7AY, UK\\
E-mail: sc3705@bath.ac.uk\\

\textsc{Nikos Katzourakis}\\
Department of Mathematics and Statistics, University of Reading\\
Whiteknights Campus, Pepper Lane, Reading RG6 6AX, UK\\
E-mail: n.katzourakis@reading.ac.uk\\

\textsc{Roger Moser}\\
Department of Mathematical Sciences, University of Bath\\
Bath BA2 7AY, UK\\
E-mail: r.moser@bath.ac.uk

\end{document}